\documentclass{amsart}
\usepackage{amssymb}
\usepackage{graphicx}
\usepackage{color}
\usepackage[import,all]{xypic}

\def\R{\mathbb R}
\def\C{\mathbb C}

\def\SU{\mathrm{SU}}

\def\SL{\mathrm{SL}}

\def\g{\mathfrak g}

\def\a{\mathfrak a}
\def\n{\mathfrak{n}}
\def\h{\mathfrak{h}}
\def\t{\mathfrak{t}}

\def\d{\mathfrak{d}}

\def\k{\mathfrak{k}}

\renewcommand\b{\mathfrak b}            

\def\bw{\mathbf w}

\def\p{\mathfrak p}          

\newtheorem{theorem}{Theorem}[section]
\newtheorem{prop}[theorem]{Proposition}
\newtheorem{corollary}[theorem]{Corollary}
\newtheorem{lemma}[theorem]{Lemma}
\newtheorem{example}{Example}

\theoremstyle{definition}

\newtheorem{notation}[theorem]{Notation}

\theoremstyle{remark}
\newtheorem{remark}[theorem]{Remark}

\newtheorem{acknowledgements}[theorem]{Acknowledgements}

\begin{document}
\title[Demazure and Bott-Samelson Resolutions]{Equivalence of Demazure and Bott-Samelson Resolutions via Factorization}
\date{\today}
\author{Arlo Caine}
\email{jacaine@cpp.edu}

\begin{abstract}
Let $G$, $B$, and $H$ denote a complex semi-simple algebraic group, a Borel subgroup of $G$, and a maximal complex torus in $B$, respectively.  Choose a compact real form $K$ of $G$ such that $T=K\cap H$ is a maximal torus in $T$.  Then there are two models for the flag space of $G$: the complex quotient $X=G/B$ and the real quotient $K/T$.  These models are smoothly equivalent via the map $\tilde{\mathbf k}\colon G/B\to K/T$ induced by factorization in $G$ relative to the Iwasawa decomposition $G=KAN$, where $N$ is the nilradical of $B$ and $H=TA$.  Likewise, there are two models for resolutions of the Schubert subvarieties $\overline{X_w}\subset X$: the Demazure resolution of $\overline{X_w}$ which is constructed via a complex algebraic quotient and the Bott-Samelson resolution of $\mathbf k(\overline{X_w})$ which is constructed as a real quotient of compact groups.  This paper makes explicit the equivalence and compatibility of these two resolutions using factorization.  As an application, we can compute the change of variables map relating the standard complex algebraic coordinates on $X_w$ to Lu's real algebraic coordinates on $\tilde{\mathbf k}(X_w)$.
\end{abstract}

\maketitle

\section{Introduction}

Let $G$ be a complex semisimple algebraic group and let $B$ be a Borel subgroup of $G$.  Let $H$ be a maximal complex torus in $B$ (and hence in $G$) and let $W=N_G(H)/H$ denote the Weyl group of $G$ with respect to $H$. The complex quotient $X=G/B$ is a complex projective variety and a model for the flag space of $G$.  Choose a compact real form $K$ of $G$ such that $T=K\cap H$ is a maximal torus in $K$.  Then write $H=TA$ for the Cartan decomposition of $H$ under the Cartan decomposition $G$ relative to $K$.  The real quotient $K/T$ is another model for the flag space of $G$.

The typical argument for the equivalence of the models $G/B$ and $K/T$ starts with the observation that the canonical inclusion $K/T\to G/B$ is an injective immersion.  Let $N$ denote the nil-radical of $B$.  Using the Iwasawa decomposition $G=KAN$, and the fact that $N$ is stable under the adjoint action of $T$, one proves that the map is surjective. Since the domain is compact, the inverse map must be smooth by an inverse function theorem argument.  Hence, the canonical inclusion $K/T\to G/B$ is a diffeomorphism.

However, the inverse of this map can be made more explicit.  Let $D=AN$ so that the Iwasawa decomposition has the form $G=KD$. Since multiplication in $G$ induces a diffeomorphism $K\times D\to G$, we know that each element $g\in G$ has a unique factorization of the form $g=\mathbf k(g)\mathbf d(g)$ where $\mathbf k(g)\in K$ and $\mathbf d(g)\in D$.  This defines smooth functions $\mathbf k\colon G\to K$ and $\mathbf d\colon G\to D$.  The map $\mathbf k$ is right $D$-invariant and right $T$-equivariant and thus induces a map $\tilde{\mathbf k}\colon G/B\to K/T$ since $B=TD$ and $D$ is stable under the adjoint action of $T$.  If $[g]_{G/B}$ denotes the class of $g\in G$ in $G/B$ and $[k]_{K/T}$ denotes the class of $k\in K$ in $K/T$, then the inclusion $\tilde{\iota}\colon K/T\to G/B$ is given by $[k]_{K/T}\mapsto [k]_{G/B}$ while $\tilde{\mathbf k}\colon G/B\to K/T$ is given by $\tilde{\mathbf k}[g]_{G/B}=[\mathbf k(g)]_{K/T}$ by definition.  Under the composition $K/T\to G/B\to K/T$ one has $[k]_{K/T}\mapsto [k]_{G/B}\mapsto [\mathbf k(k)]_{K/T}=[k]_{K/T}$ since $\mathbf k$ restricts to the identity on $K$.  Thus, $\tilde{\mathbf k}$ is the inverse of the canonical inclusion $K/T\to G/B$.

In a similar fashion, this paper makes explicit the equivalence of two different models of resolutions of Schubert varieties $\overline{BwB}\subset X=G/B$.  Recall that the group $B$ acts on $G$ from the left by multiplication and this induces a left action of $B$ on $X$.  The orbits of $B$ on $X$ are finite in number and indexed by the elements of $W$.  Each orbit $X_w=BwB$, for $w\in W$, is called a Schubert cell in $X$ and is a complex subvariety of $X$ isomorphic to $\C^{\ell(w)}$ where $\ell(w)$ is the length of $w$.  Its closure in $X$, denoted $\overline{X_w}$, is called a Schubert variety.

In general, the Schubert varieties are singular.  There are two similar constructions which produce smooth resolutions of these spaces depending on a reduced decomposition of $w$ into a sequence $\bw$ of simple reflections.
\begin{enumerate}
\item The Bott-Samelson resolution \cite{BottSamelson} considers the Schubert variety as a compact topological subspace $\tilde{\mathbf k}(\overline{X_w})$ of $K/T$ and uses the subgroups of $K$ of minimal rank associated to the simple reflections in $\bw$ to construct the resolving space $\mathcal{BS}_\bw$.
\item The Demazure resolution \cite{Demazure} considers $\overline{X_w}$ itself as a complex algebraic subvariety of $X=G/B$ and uses the minimal parabolic subgroups of $G$ containing $B$ associated to the simple reflections in $\bw$ to construct the resolving space $\mathcal{D}_\bw$.
\end{enumerate}
There is a canonical inclusion $\mathcal{BS}_\bw\to \mathcal{D}_\bw$ which one can argue is a proper bijective immersion. Since the domain is compact, it is possible to conclude, abstractly, that the map is a diffeomorphism by an inverse function theorem argument.  However, as with the canonical inclusion $K/T\to G/B$, the inverse of this inclusion can be made more explicit.  This is the main point of Section \ref{EquivFactSection} and a new result.  In this setting however, the inverse is induced from a map built from both of the factorization maps $\mathbf k$ and $\mathbf d$ together with the multiplication map on $G$.  As an application, we show in Section \ref{ApplicationSection} how to use this equivalence to compute the change of coordinates between the standard holomorphic coordinates on $X_w$ and Lu's real algebraic coordinates on $\tilde{\mathbf k}(X_w)$.

\begin{acknowledgements}
This research was supported by the Provost's Teacher-Scholar program at California State Polytechnic University Pomona.  The author is grateful for this support and would also like to thank Doug Pickrell and Sam Evens for useful discussions.
\end{acknowledgements}

\section{Demazure vs. Bott-Samelson Resolutions\label{DvsBSresolutionsSection}}

The point of this section is review the constructions of the Demazure and Bott-Samelson resolutions and establish notation that will be used in the proof of the main results in Section \ref{EquivFactSection}.  Let $\g$ denote the Lie algebra of $G$, and let $\b$ and $\h$ denote the subalgebras of $\g$ corresponding to $B$ and $H$.  Then $\h$ is a Cartan subalgebra of $\g$ and we can decompose $\g$ into root spaces under the adjoint action of $\h$.  The Borel subalgebra $\b$ determines a choice of positive roots and we let $\Delta$ denote the set of simple positive roots.  The parabolic subalgebras of $\g$ containing $\b$ are in one-to-one correspondence with the subsets of $\Delta$.  The minimal such parabolic subalgebras (other than $\b$ itself) are those corresponding to the singleton subsets of $\Delta$.  Recall that each simple positive root determines a unique simple reflection in $W=N_G(H)/H$.  It will be convenient to write $\p_s$ for the minimal parabolic subalgebra of $\g$ containing $\b$ determined by the simple positive root whose associated simple reflection is $s\in W$.  Let $P_s$ denote the corresponding parabolic subgroup of $G$ containing $B$.

Let $w\in W$ be given and suppose that $\bw=(s_1,s_2,\dots,s_\ell)$ is a finite sequence of simple reflections associated to positive simple roots such that $w=s_1s_2\dots s_\ell$.  Note that the subscript $i$ in $s_i$ indicates its position in the sequence, not that it is the $i^{th}$ simple positive root in some fixed enumeration of those roots.  Then $B^\ell$ acts freely from the right on
\begin{equation}
P_\bw=P_{s_{1}}\times P_{s_{2}}\times \dots \times P_{s_{\ell}}
\end{equation}
by the action
\[
(p_1,p_2,\dots,p_\ell).(b_1,b_2,\dots,b_\ell)=(p_1b_1,b_1^{-1}p_2b_2,\dots,b_{\ell-1}^{-1}p_\ell b_\ell)
\]
and we denote the quotient $P_\bw/B^\ell$ by
\begin{eqnarray}
\mathcal {D}_{\bw} & = & P_\bw/B^\ell \\
& = & P_{s_{1}}\times_{B} P_{s_{2}} \times_{B} \dots \times_{B} P_{s_{\ell}}/B
\end{eqnarray}
for the total space of the Demazure resolution determined by $\bw$.  It is a smooth complex projective variety.

Multiplication $P_\bw=P_{s_{1}}\times P_{s_{2}}\times \dots \times P_{s_{\ell}}\to G$ is equivariant for the right action of $B^\ell$ on $P_{\bw}$ and the right action of $B$ on $G$ and thus induces a complex anlaytic map
\[
\rho_\bw\colon \mathcal{D}_\bw\to X=G/B.
\]
When the sequence $\bw$ is reduced, i.e., $\ell=\ell(w)$, the image of $\rho_\bw$ is the Schubert variety $\overline{X_w}$.  Each decomposition of $w$ as a reduced sequence $\bw$ of simple reflections thus determines a resolution of singularities for $\overline{X_w}$ restricting to an isomorphism $\rho_\bw^{-1}(X_w)\to X_w$ over the Schubert cell $X_w$.

Alternatively, if one chooses a compact real form $\k$ of $\g$ such that $\t=\k\cap\h$ is a Cartan subalgebra of $\k$, then the intersections $K_{s_i}=K\cap P_{s_i}$, for $i=1,2,\dots,\ell$ are the subgroups of $K$ of minimal rank containing $T$.  These are the subgroups of $K$ containing $T$ with roots $\alpha_i$ and $-\alpha_i$. Let
\begin{equation}
K_\bw=K_{s_1}\times K_{s_2}\times\dots\times K_{s_\ell}
\end{equation}
and note that $K_\bw$ is a real compact subgroup of $P_\bw$.  The free action of $B^\ell$ on $P_\bw$ restricts to a free action of $T^\ell$ on $P_\bw$ which stabilizes $K_\bw$, since $T$ is a subgroup of $K_{s_i}$ for each $i=1,2,\dots,\ell$.  We denote the quotient by
\begin{eqnarray}
\mathcal{BS}_\bw & = & K_\bw/T^\ell  \\
&  = & K_{s_1}\times_T K_{s_2}\times_T \dots \times_T K_{s_\ell}/T
\end{eqnarray}
for the Bott-Samelson resolution.  It is a smooth compact manifold.

Multiplication $K_\bw=K_{s_1}\times K_{s_2}\times \dots \times K_{s_\ell}\to K$ is equivariant for the right action of $T^\ell$ on $K_\bw$ and the right action of $T$ on $K$ and thus induces a smooth map
\[
\rho_\bw^K\colon \mathcal{BS}_\bw \to K/T.
\]
When the sequence $\bw$ is reduced, the image of this map is $\tilde{\mathbf k}(\overline{X_w})$ in $K/T$.  Each decomposition of $w$ as a reduced sequence $\bw$ of simple reflections thus determines a smooth manifold $\mathcal{BS}_\bw$ and a smooth map $\rho_\bw^K\colon \mathcal{BS}_\bw\to \tilde{\mathbf{k}}(\overline{X_w})$ which is a diffeomorphism onto $\tilde{\mathbf k}(X_w)$ when restricted to the pre-image of $\tilde{\mathbf k}(X_w)$.

Just as the inclusion $K\to G$ induces a canonical inclusion $K/T\to G/B$, the inclusion $K_\bw\to P_\bw$ is equivariant for the actions of $T^\ell$ on $K_\bw$ and $B^\ell$ on $P_\bw$ and thus induces a canonical inclusion $\mathcal{BS}_\bw\to \mathcal D_\bw$.  In the next section, we construct an explicit inverse for this map using factorization.

\section{Equivalence of Resolutions via Factorization\label{EquivFactSection}}

Let $\k_{s_i}=\k\cap \p_{s_i}$ denote the Lie algebra of $K_{s_i}$ for each $i=1,2,\dots,\ell$.  Then $\p_{s_i}=\k_{s_i}+\d$ is an Iwasawa decomposition of $\p_{s_i}$ and therefore the multiplication map $K_{s_i}\times D\to P_{s_i}$ is a diffeomorphism for each $i=1,2,\dots,\ell$.  Thus, $\mathbf k\colon G\to K$ restricts to a map $\mathbf k\colon P_{s_i}\to K_{s_i}$ for each $i=1,2,\dots,\ell$ and the product map $\mathbf k^\ell\colon G^\ell\to K^\ell$ then restricts to a map $\mathbf k^\ell\colon P_\bw\to K_\bw$.

\begin{notation}\label{Defn_of_phi_m}
Let $\phi_\ell\colon P_\bw\to K_\bw$ be defined by
\begin{equation}
\phi_\ell(p_1,p_2,\dots,p_\ell)=(\mathbf k(q_1),\mathbf k(q_2),\dots,\mathbf k(q_\ell))
\end{equation}
where $(q_1,q_2,\dots,q_\ell)\in P_\bw$ is obtained from $(p_1,p_2,\dots,p_\ell)\in P_\bw$ recursively by $q_1=p_1$ and $q_k=\mathbf d(q_{k-1})p_k$ for $k=2,3,\dots, \ell$.
\end{notation}

Then $\phi_\ell$ is smooth since its components are compositions of factorizations and multiplications in $G$.  Note that the correspondence $\beta_\ell\colon P_\bw\to P_\bw$ defined by $\beta(p_1,p_2,\dots,p_\ell)=(q_1,q_2,\dots,q_\ell)$ is a real algebraic diffeomorphism $P_\bw\to P_\bw$.  It is clearly a smooth map, since its components are compositions of factorizations and multiplications in $G$, but its recursive definition shows that it can easily be inverted and that the inverse is also smooth, involving compositions of factorizations, multiplications, and inversions in $G$.

\begin{lemma} The smooth map
$\phi_\ell\colon P_\bw\to K_\bw$ intertwines the actions of $B^\ell$ on $P_\bw$ and $T^\ell$ on $K_\bw$ and thus descends to a smooth map
\[
\tilde{\phi}_\ell\colon \mathcal D_\bw\to \mathcal BS_\bw.
\]
\end{lemma}
\begin{proof}We need to prove that given $(p_1,p_2,\dots,p_\ell)\in P_\bw$ and $(b_1,b_2,\dots,b_\ell)\in B^\ell$ there exists $(t_1,t_2,\dots,t_\ell)\in T^\ell$ such that
\begin{equation}\label{NTSforLemma32}
\phi_\ell((p_1,p_2,\dots,p_\ell).(b_1,b_2,\dots,b_\ell))=\phi_\ell(p_1,p_2,\dots,p_\ell).(t_1,t_2,\dots,t_\ell).
\end{equation}
Set $(t_1,t_2,\dots,t_\ell)=(\mathbf k(b_1),\mathbf k(b_2),\dots,\mathbf k(b_\ell))$.  In the proof that this choice satisfies (\ref{NTSforLemma32}), it will be convenient to write $(p_k)'$ for the $k^{th}$ coordinate of $\beta(p_1,p_2,\dots,p_\ell)$. Then, the $k^{th}$ coordinate of $\beta((p_1,p_2,\dots,p_\ell).(b_1,b_2,\dots,b_\ell))$ is written $(b_{k-1}^{-1}p_kb_k)'$ while $(p_k)'=q_k$ in our previous notation.  The proof will require two auxiliary identities
\begin{equation}\label{aux1}
\mathbf k((b_{k-1}^{-1}p_{k}b_{k})')=t_{k-1}^{-1}\mathbf k(q_k)t_k, \ k=1,2,\dots,\ell
\end{equation}
and
\begin{equation}\label{aux2}
\mathbf d((b_{k-1}^{-1}p_kb_k)')=t_{k-1}^{-1}\mathbf d(q_k)b_k,\ k=1,2,\dots,\ell
\end{equation}
where $b_0=t_0$ denotes the identity in $G$.  We will prove (\ref{aux2}) by induction on $k$ and deduce (\ref{aux1}) along the way.  The $k=1$ case of (\ref{aux1}) is clear because the first coordinate of $\beta$ is the identity map on $P_{s_1}$.

For the basis step of (\ref{aux2}) we use the facts that $\mathbf d$ is right $D$-equivariant, but converts right multiplication by $T$ into conjugation by $t^{-1}$.  Indeed, if $g\in G$ and $d\in D$ then $gd=\mathbf k(g)\mathbf d(g)d$ implies that $\mathbf d(gd)=\mathbf d(g)d$ and if $t\in T$ then $gt=\mathbf k(g)\mathbf d(g)t=\mathbf k(g)t\cdot t^{-1}\mathbf d(g)t$ implies that $\mathbf d(gt)=t^{-1}\mathbf d(g)t$ because $D$ is $\mathrm{Ad}(T)$-stable.  The base case of (\ref{aux2}) is then verified by $\mathbf d((p_1b_1)')=\mathbf d(p_1b_1)=\mathbf d(p_1t_1\mathbf d(b_1))=t_1^{-1}\mathbf d(p_1)t_1\mathbf d(b_1)=t_1^{-1}\mathbf d(q_1)b_1$ since the first coordinate of $\beta$ is the identity map on $P_{s_1}$.

For the inductive step, assume that (\ref{aux2}) holds for some $k$.  Then
\begin{eqnarray*}
(b_k^{-1}p_{k+1}b_{k+1})' & = & \mathbf d((b_{k-1}^{-1}p_kb_k)')\cdot b_k^{-1}p_{k+1}b_{k+1} \\
& = & t_k^{-1}\mathbf d(q_k)b_k \cdot b_k^{-1}p_{k+1}b_{k+1} \\
& = & t_k^{-1}\mathbf d(q_k)p_{k+1}b_{k+1}
\end{eqnarray*}
using the recursive definition of the $(k+1)^{st}$ coordinate of $\beta$ and the inductive assumption.  Now we recall that $q_{k+1}=\mathbf d(q_k)p_{k+1}$, so we can substitute, factor, and insert to find
\begin{eqnarray*}
(b_k^{-1}p_{k+1}b_{k+1})' & = & t_k^{-1}q_{k+1}b_{k+1} \\
 & = & t_k^{-1} \mathbf k(q_{k+1})\mathbf d(q_{k+1})b_{k+1} \\
 & = & t_k^{-1} \mathbf k(q_{k+1})t_{k+1}\cdot t_{k+1}^{-1}\mathbf d(q_{k+1})t_{k+1}\mathbf d(b_{k+1}).
\end{eqnarray*}
From this last factorization, it follows that $\mathbf k((b_k^{-1}p_{k+1}b_{k+1})')=t_k^{-1} \mathbf k(q_{k+1})t_{k+1}$, the $k+1$ case of (\ref{aux1}) and $\mathbf d((b_k^{-1}p_{k+1}b_{k+1})')=t_{k+1}^{-1}\mathbf d(q_{k+1})t_{k+1}\mathbf d(b_{k+1})=t_{k+1}^{-1}\mathbf d(q_{k+1})b_{k+1}$ which is the $k+1$ case of (\ref{aux2}).  This proves (\ref{aux1}) and (\ref{aux2}).

Finally, by applying (\ref{aux1}), we have
\begin{eqnarray*}
\phi_\ell((p_1,p_2,\dots,p_\ell).(b_1,b_2,\dots,b_\ell))& = & (\mathbf k((p_1b_1)'),\mathbf k((b_1^{-1}p_2b_2)'),\dots,\mathbf k((b_{m-1}^{-1}p_\ell b_\ell)') \\
& = & (\mathbf k(q_1)t_1,t_1^{-1}\mathbf k(q_2)t_2,\dots,t_{m-1}^{-1}\mathbf k(q_\ell)t_\ell) \\
& = & \phi_\ell(p_1,p_2,\dots,p_\ell).(t_1,t_2,\dots,t_\ell)
\end{eqnarray*}
as was to be shown.
\end{proof}

The main result of this paper is the following.

\begin{theorem}\label{Equivalence_of_Constructions}
The induced map $\tilde{\phi}_\ell\colon \mathcal{D}_\bw\to \mathcal{BS}_\bw$ is a diffeomorphism whose inverse is the canonical inclusion $\mathcal{BS}_\bw\to \mathcal D_\bw$.
\end{theorem}

\begin{proof}Let $\tilde{\iota}\colon \mathcal{BS}_\bw\to \mathcal{D}_\bw$ denote the canonical inclusion.  To prove the theorem, we will argue that $\tilde{\phi}_\ell\circ\tilde{\iota}=\mathrm{id}_{\mathcal{BS}_\bw}$ and $\tilde{\iota}\circ\tilde{\phi}_\ell=\mathrm{id}_{\mathcal D_\bw}$.   On the one hand, if $(k_1,k_2,\dots,k_\ell)\in K_\bw$ then
\[
(\phi_\ell\circ \iota)(k_1,k_2,\dots,k_\ell)=(\mathbf k(q_1),\mathbf k(q_2),\dots, \mathbf k(q_\ell))
\]
where we recall that $q_1=k_1$ and $q_j=\mathbf d(q_{j-1})k_j$ for $j=2,3,\dots,\ell$.  But if $q_{j-1}\in K$, then $\mathbf d(q_{j-1})=1\in D$ and thus $q_j=k_j\in K$.  Since $q_1=k_1$, we see that
\[
(\phi_\ell\circ\iota)(k_1,k_2,\dots,k_\ell)=(k_1,k_2,\dots,k_\ell)
\]
by finite induction.  Thus $\tilde{\phi}_\ell\circ\tilde{\iota}=\mathrm{id}_{\mathcal{BS}_\bw}$.

On the other hand, given $(p_1,p_2,\dots,p_\ell)\in P_\bw$ we recursively compute the sequence $(q_1,q_2,\dots,q_\ell)$ and set $b_j=\mathbf d(q_j)^{-1}$ for $j=1,2,\dots,\ell$.  Then $\mathbf k(q_1)=\mathbf k(p_1)=p_1\mathbf d(p_1)^{-1}=p_1b_1$.  Furthermore, for $2\le j\le \ell$ we have
\[
\mathbf k(q_j)=q_j\mathbf d(q_j)^{-1} = \mathbf d(q_{j-1})p_j\mathbf d(q_j)^{-1} = b_{j-1}^{-1}p_jb_j
\]
which means that
\[
(\iota\circ\phi_\ell)(p_1,p_2,\dots,p_\ell)=(\mathbf k(q_1),\mathbf k(q_2),\dots,\mathbf k(q_\ell))=(p_1,p_2,\dots,p_\ell).(b_1,b_2,\dots,b_\ell)
\]
with respect to the action of $B^\ell$.  Thus, $\tilde{\iota}\circ\tilde{\phi}_\ell=\mathrm{id}_{\mathcal{D}_\bw}$.
\end{proof}

\begin{corollary}
The diagram
\begin{equation}\label{resolutions_commute}
\xymatrix{ \mathcal D_\bw \ar[r]^{\rho_\bw} \ar[d]_{\tilde{\phi}_\ell} & G/B \ar[d]^{\tilde{\mathbf k}} \\
\mathcal{BS}_\bw \ar[r]^{\rho_\bw^K} & K/T
}
\end{equation}
commutes.
\end{corollary}
\begin{proof}
The commutative diagram on the left in (\ref{diagrams}) induces the commutative diagram on the right.
\begin{equation}\label{diagrams}
\xymatrix{
P_\bw \ar[r]^{\mathrm{mult}} &  G  & & \mathcal D_\bw \ar[r]^{\rho_\bw} & G/B \\
K_\bw \ar[r]^{\mathrm{mult}} \ar[u]^{\iota} & K \ar[u]_{\iota} & & \mathcal{BS}_\bw \ar[r]^{\rho_\bw^K} \ar[u]_{\tilde{\iota}} & K/T \ar[u]_{\tilde{\iota}}
}
\end{equation}
The map $\tilde{\mathbf k}$ is the inverse to the inclusion $\tilde{\iota}\colon K/T\to G/B$ and the map $\tilde{\phi}_\ell$ is the inverse to the inclusion $\tilde{\iota}\colon \mathcal{BS}_\bw \to \mathcal D_\bw$ by Theorem \ref{Equivalence_of_Constructions}.  Thus, inverting the vertical arrows of the diagram on the right yields the result.
\end{proof}

\begin{remark}
In particular, when $\bw$ is a reduced decomposition of $w$ (so that $\ell=\ell(w)$) the map $\rho_\bw\colon \mathcal D_\bw\to \overline{X_w}$ is the Demazure resolution and $\rho_\bw^K\colon \mathcal{BS}_\bw\to \tilde{\mathbf k}(\overline{X_w})$ is the Bott-Samelson resolution.  Then $\tilde{\phi}_\ell$ gives an explicit equivalence of resolutions via factorization and the following diagram commutes.
\begin{equation}\label{resolutions_commute2}
\xymatrix{ \mathcal D_\bw \ar[r]^{\rho_\bw} \ar[d]_{\tilde{\phi}_\ell} & \overline{X_w} \ar[d]^{\tilde{\mathbf k}} \\
\mathcal{BS}_\bw \ar[r]^{\rho_\bw^K} & \tilde{\mathbf k}(\overline{X_w})
}
\end{equation}
\end{remark}

\section{Coordinates on $X_w$\label{ApplicationSection}}

As an application, we indicate how this explicit equivalence between the Demazure and Bott-Samelson resolutions can be used to compute the change of variables between standard holomorphic coordinates $(\zeta_1,\zeta_2,\dots,\zeta_\ell)$ on the complex model $X_w$ of the Schubert cell and Lu's $C^\infty$-coordinates $(z_1,\bar{z}_1,z_2,\bar{z}_2,\dots,z_\ell,\bar{z}_\ell)$ on the real model $\tilde{\mathbf k}(X_w)$ from \cite{Lu}.

\subsection{Holomorphic Coordinates on $X_w$}
First, we review the standard construction of holomorphic coordinates on $X_w$.  Let $B^-$ denote the Borel subgroup of $G$ such that $B^-\cap B=H$ and let $N^-$ denote the nil-radical of $B^-$.  For each $w\in W$, set $N_w=N\cap wN^-w^{-1}$.  Then the orbit map $N_w\to X_w\subset G/B$ defined by $n\mapsto nwB$ is a biholomorphism because for each representative $\dot w$, the subset $N_w\dot w\subset G$ is transverse to the action of $B$.  One can thus obtain holomorphic coordinates on $X_w$ from a holomorphic parameterization of the nilpotent group $N_w$.  There are many ways to do this, but the construction of the Demazure resolution of $\overline{X_w}$ suggests an explicit algebraic procedure.

Let $\langle\cdot,\cdot\rangle$ denote the Killing form for $\g$.  For each positive root $\alpha$, let $H_\alpha$ denote the coroot defined by $\langle H_\alpha,H\rangle=\alpha(H)$ for each $H\in\h$. For each simple positive root $\gamma$, choose root vectors $E_\gamma$ and $E_{-\gamma}$ for $\gamma$ and $-\gamma$, respectively, such that $\langle E_\gamma,E_{-\gamma}\rangle =1$.  Then $[E_{\gamma},E_{-\gamma}]=H_\gamma$. Set
\[\textstyle{
\check{H}_\gamma=\frac{2}{\langle\langle \gamma,\gamma\rangle\rangle}H_\gamma,\ \check{E}_\gamma=\sqrt{\frac{2}{\langle\langle \gamma,\gamma\rangle\rangle}}E_\gamma, \ \check{E}_{-\gamma}=\sqrt{\frac{2}{\langle\langle \gamma,\gamma\rangle\rangle}}E_{-\gamma}}
\]
where $\langle\langle\cdot,\cdot\rangle\rangle$ denotes the dual of the Killing form.
Then the linear map $\mathrm{sl}(2,\C)\to\g$ determined by
\[
\begin{pmatrix}
1 & 0 \\ 0 & -1
\end{pmatrix}\mapsto \check{H}_\gamma,\ \begin{pmatrix}
0 & 1 \\ 0 & 0
\end{pmatrix}\mapsto \check{E}_\gamma,\ \begin{pmatrix}
0 & 0 \\ 1 & 0
\end{pmatrix}\mapsto \check{E}_{-\gamma}
\]
is a Lie algebra homomorphism which integrates to a Lie group homomorphism
\[
\Psi_\gamma\colon \SL(2,\C)\to G.
\]
It will be convenient to think of $\gamma$ as corresponding to a unique simple reflection $s$ and denote $\Psi_\gamma$ by $\Psi_s$.  Note that by construction, $\Psi_s$ maps diagonal matrices into $H$ and unipotent upper triangular matrices into $N_s$.

The matrix
\[
\sigma=\begin{pmatrix} 0 & i \\ i & 0 \end{pmatrix}
\]
is a representative for the non-trivial element of the Weyl group of $\mathrm{SL}(2,\C)$ relative to the Cartan subalgebra of diagonal matrices.  Thus, to the decomposition $\bw=(s_1,s_2,\dots,s_\ell)$ for $w$, we can associate representatives
\[
\dot s_j=\Psi_{s_j}(\sigma), \text{ and }\dot w=\dot s_1\dot s_2\cdots \dot s_\ell
\]
in $N_G(H)$ for each reflection $s_j$ in the sequence and for the element $w$. We will write $\dot\bw=(\dot s_1,\dot s_2,\dots,\dot s_\ell)\in K_\bw\subset P_\bw$. For $\zeta_j\in\C$ we define
\[
n_{\zeta_j}=\Psi_{s_j}\begin{pmatrix} 1 & \zeta_j \\ 0 & 1\end{pmatrix}=\exp(\zeta_j\check{E}_{\gamma_j})\in N_{s_j}
\]
where $\gamma_j$ is the simple positive root associated to the $j^{th}$ reflection $s_j$ in the sequence $\bw$.  The function $\C\to N_{s_j}$ defined by $\zeta_j\mapsto n_{\zeta_j}$ is a biholomorphism.  Let $N_\bw=N_{s_1}\times N_{s_2}\times \dots \times N_{s_\ell}$.  Then $(\zeta_1,\dots,\zeta_\ell)\to (n_{\zeta_1},\dots,n_{\zeta_\ell})$ is a parametrization of $N_\bw$ by $\C^\ell$. The following result is a standard fact.

\begin{prop}\label{hol_coords}
If $\bw$ is a reduced decomposition of $w$, then
there exists a unique biholomorphism $M_{\dot \bw}\colon N_\bw\to N_w$ such that
\[
M_{\dot \bw}(n_{\zeta_1},\dots,n_{\zeta_\ell})\dot w=n_{\zeta_1}\dot s_1n_{\zeta_2}\dot s_2\dots n_{\zeta_\ell}\dot s_\ell
\]
for each $(\zeta_1,\dots,\zeta_\ell)\in\C^\ell$.
\end{prop}

Thus, if we write $N_\bw\dot\bw$ for the orbit $N_{s_1}\dot s_1\times \dots\times N_{s_\ell}\dot s_\ell$ of $\dot \bw$ in $P_\bw$ under left multiplication by $N_\bw$, we obtain the following commutative diagram in which each arrow is a biholomorphism.
\[
\xymatrix{
N_\bw \ar[r]^{M_{\dot\bw}} \ar[d]_{.\dot\bw} & N_w \ar[d]^{.\dot{w}} \\
N_\bw\dot\bw \ar[r]^{\mathrm{mult}} \ar[d]_{\mathrm{mod}\,B^\ell} & N_w\dot w \ar[d]^{\mathrm{mod}\,B} \\
N_\bw\dot\bw.B^\ell \ar[r]^{\rho_\bw} & N_w\dot w.B
}
\]
The map $h_{\dot\bw}\colon \C^\ell\to \mathcal D_\bw$ defined by
\begin{equation}\label{defn_of_h_w}
h_{\dot\bw}(\zeta_1,\dots,\zeta_\ell)=[n_{\zeta_1}\dot s_1,\dots,n_{\zeta_\ell}\dot s_\ell]\in \mathcal D_{\bw}
\end{equation}
is then a biholomorphism onto the open set $N_\bw\dot\bw.B^\ell=\rho_\bw^{-1}(X_w)$ giving holomorphic coordinates on the Demazure resolution.  The composition $\rho_\bw\circ h_{\dot\bw}\colon \C^\ell\to X_w$ is a holomorphic coordinate chart on the complex model $X_w\subset G/B$ of the Schubert cell.

\subsection{Lu's coordinates on $\mathbf k(X_w)$}
In (\cite{Lu}, Theorem 1, pg. 360) Lu proved an analog of the previous proposition, giving rise to a different coordinate system on the real model $\mathbf k(X_w)\subset K/T$ of the Schubert cell.  Although it is convenient to express the formulas for these coordinates in terms of complex variables $z_j$, the coordinates are not holomorphic, in general, as will be made clear below.

First, we must choose a specific compact real form $\k$ of $\g$ as follows.   For each remaining positive root $\alpha$ choose root vectors $E_{\alpha}$ and $E_{-\alpha}$ for $\alpha$ and $-\alpha$, respectively, such that $\langle E_{\alpha},E_{-\alpha}\rangle =1$.  Then $[E_{\alpha},E_{-\alpha}]=H_{\alpha}$. Set
\[
X_\alpha=E_\alpha-E_{-\alpha}, \ Y_\alpha=i(E_\alpha+E_{-\alpha})
\]
for positive root $\alpha$.  Then the real subspace $\k=\mathrm{span}_\R\{iH_\alpha,X_\alpha,Y_\alpha:\alpha>0\}$ of $\g$ is a compact real form of $\g$ such that $\t=\mathrm{span}_\R\{iH_\alpha:\alpha>0\}$ is a Cartan subalgebra of $\k$. In this setting, $\a=i\t$.  We let $K$, $T$, and $A$ denote the real connected subgroups of $G$ integrating $\k$, $\t$, and $\a$, respectively.  This specifies factorization maps $\mathbf k\colon G\to K$ and $\mathbf d\colon G\to D=AN$.

\begin{prop}[Lu]\label{Lu_coords}
If $\bw$ is a reduced decomposition of $w$ then there exists a unique diffeomorphism $F_{\dot \bw}\colon N_\bw\to N_w$ such that
\[
\mathbf k(F_{\dot \bw}(n_{z_1},n_{z_2},\dots,n_{z_\ell})\dot w)=\mathbf k(n_{z_1}\dot s_1)\mathbf k(n_{z_2}\dot s_2)\cdots \mathbf k(n_{z_\ell}\dot s_\ell)
\]
for each $(z_1,z_2,\dots,z_\ell)\in\C^\ell$.
\end{prop}
As a consequence, we then obtain a commutative diagram
\[
\xymatrix{
\mathbf k^\ell(N_\bw\dot\bw).T^\ell \ar[r]^{\rho_\bw^K} & \mathbf k(N_w\dot w).T \\
\mathbf k^\ell(N_\bw\dot\bw) \ar[r]^{\mathrm{mult}} \ar[u]^{\mathrm{mod}\,T^\ell} & \mathbf k(N_w\dot w) \ar[u]_{\mathrm{mod}\,T} \\
N_\bw \ar[r]^{F_{\dot\bw}} \ar[u]^{\mathbf k^\ell(\cdot\dot\bw)} & N_w \ar[u]_{\mathbf k(\cdot\dot w)}
}
\]
in which each arrow is a diffeomorphism.  Recall that $\mathbf k(N_w\dot w).T=\tilde{\mathbf k}(X_w)$.  Thus, the map $j_{\dot\bw}\colon \C^\ell\to \mathcal{BS}_\bw$ defined by
\begin{equation}\label{defn_of_j_bw}
j_{\dot\bw}(z_1,z_2,\dots,z_\ell)=[\mathbf k(n_{z_1}\dot s_1),\dots,\mathbf k(n_{z_\ell}\dot s_\ell)]\in\mathcal{BS}_\bw
\end{equation}
is a diffeomorphism onto the open set $\mathbf k^\ell(N_\bw\dot\bw).T^\ell=(\rho_\bw^K)^{-1}(\tilde{\mathbf k}(X_w))$, giving smooth coordinates on the Bott-Samelson resolution $\mathcal{BS}_\bw$.  Furthermore, the composition $\rho_\bw^K\circ j_{\dot\bw}\colon \C^\ell\to \tilde{\mathbf k}(X_w)$ is a diffeomorphism onto $\tilde{\mathbf k}(X_w)\subset K/T$ of the Schubert cell. These are Lu's coordinates on the real model of the Schubert cell.

Despite the fact that Lu's construction involves the $\mathbf k$ factorization map, which is in general very complicated, the coordinates can be explicitly computed. A key role is played by the real function $a\colon \C\to \R$ defined by
\begin{equation}
a(z)=(1+|z|^2)^{-1/2}
\end{equation}
which occurs in the following Lemma.

\begin{lemma}\label{k_and_d_lemma}
For each $z\in \C$,
\[
\mathbf k(n_z\dot s_j)=\Psi_{s_j}\begin{pmatrix}iza(z) & ia(z) \\ ia(z) & -i\bar{z}a(z)\end{pmatrix}
\]
and
\[
\mathbf d(n_z\dot s_j)=\Psi_{s_j}\begin{pmatrix}1 & \bar{z} \\ 0 & 1\end{pmatrix}\Psi_{s_j}\begin{pmatrix}a(z)^{-1} & 0 \\ 0 & a(z)\end{pmatrix}=\exp(\bar{z}\check{E}_{\gamma_j})a(z)^{-\check{H}_{\gamma_j}}.
\]
\end{lemma}
\begin{proof}
With this choice of $K$,
\[
\Psi_{s_j}(\mathrm{SU}(2))\subset K_{s_j},\
\Psi_{s_j}\begin{pmatrix} 1 & z \\ 0 & 1 \end{pmatrix}=\exp(z\check{E}_{\gamma_j})\in N_{s_j},
\]
for each $z\in \C$ and
\[
\Psi_{s_j}\begin{pmatrix} a & 0 \\ 0 & a^{-1}\end{pmatrix}=a^{\check{H}_{\gamma_j}}\in A
\]
for each $a>0$ in $\R$. As a result, $\mathbf k(n_{z_j}\dot s_j)$ can be computed explicitly.  First, we observe that for each $z\in\C$
\[
n_{z}\dot s_j=\Psi_{s_j}\begin{pmatrix} 1 & z \\ 0 & 1\end{pmatrix}\Psi_{s_j}\begin{pmatrix}0 & i \\ i & 0\end{pmatrix}=\Psi_{s_j}\begin{pmatrix}iz & i \\ i & 0\end{pmatrix}.
\]
The equation
\[
\begin{pmatrix}iz & i \\ i & 0\end{pmatrix}=\begin{pmatrix}iza(z) & ia(z) \\ ia(z) & -i\bar{z}a(z)\end{pmatrix}\begin{pmatrix}1 & \bar{z} \\ 0 & 1\end{pmatrix}\begin{pmatrix}a(z)^{-1} & 0 \\ 0 & a(z)\end{pmatrix},
\]
where $a(z)=(1+|z|^2)^{-1/2}$, is a factorization in $\mathrm{SL}(2,\C)$ into a product of an element of $\mathrm{SU}(2)$, an upper triangular unipotent matrix, and a diagonal matrix with positive diagonal entries.  By construction, $\Psi_{s_j}$ carries this factorization into an Iwasawa decomposition of $n_z\dot s_j$.  The result follows.
\end{proof}

\subsection{Change of Coordinates}

Specializing the commutative diagram (\ref{resolutions_commute2}) to the open sets $N_\bw \dot{\bw}.B^\ell\subset \mathcal D_\bw$ and $\mathbf{k}^\ell(N_\bw\dot{\bw}).T^\ell\subset \mathcal{BS}_\bw$, and including their parameterizations $h_\bw$ from (\ref{defn_of_h_w}) and $j_\bw$ from (\ref{defn_of_j_bw}),  we obtain a commutative diagram
\begin{equation}\label{basic_diagram}
\xymatrix{\C^\ell \ar[r]^{h_\bw\,\,\,\,\,} & N_\bw\dot\bw.B^\ell \ar[d]_{\,\,\,\tilde{\phi}_\ell} \ar[r]^{\rho_{\bw}} & X_w \ar[d]^{\mathbf k} \\
\C^\ell \ar[r]^{j_\bw\,\,\,\,\,\,\,} & \mathbf k^\ell (N_\bw\dot\bw).T^\ell\ar[r]^{\,\,\,\rho_\bw^K} & \mathbf k(X_w)
}
\end{equation}
in which each arrow is a diffeomorphism. Thus $j_{\dot\bw}^{-1}\circ \tilde{\phi}_\ell\circ h_{\dot\bw}\colon \C^\ell\to\C^\ell$ gives the change of variables between the holomorphic coordinates and Lu's coordinates on the Schubert cell.

\begin{example}\label{Example1}
When $\ell(w)=1$, so that $w=s$ for some simple reflection $s$ in $W$, there is only one possible reduced sequence $\bw=(s)$ and $(\tilde{\phi}_1\circ h_{\dot\bw})(\zeta)=[\mathbf k(n_{\zeta}\dot s)]=j_{\dot\bw}(z)$ if only if $z=\zeta$.  Thus, the change of variables is the identity map.
\end{example}
\begin{remark}
For these cases, Lu's coordinate agrees with the standard holomorphic coordinate on the Schubert cell. The difference in constructions is simply one of perspective, using the real model $\mathbf k(X_s)$ in one case and the complex model $X_s$ in the other.
\end{remark}

It is possible to compute the change of variables explicitly for higher length cases because of the recursive definition $\phi_\ell$.  Set $(p_1,p_2,\dots,p_\ell)=(n_{\zeta_1}\dot s_1,n_{\zeta_2}\dot s_2,\dots, n_{\zeta_\ell}\dot s_\ell)$ in $P_\bw$.  We want to determine $(z_1,z_2,\dots,z_\ell)$ as functions of $(\zeta_1,\zeta_2,\dots,\zeta_\ell)$ such that $\tilde{\phi}_\ell[p_1,p_2,\dots,p_\ell]=[\mathbf k(n_{z_1}\dot s_1),\mathbf k(n_{z_2}\dot s_2),\dots,\mathbf k(n_{z_\ell})]$.  But $\phi_\ell(p_1,p_2,\dots,p_\ell)=(\mathbf k(q_1),\mathbf k(q_2),\dots,\mathbf k(q_\ell))$ where $(q_1,q_2,\dots,q_\ell)\in P_\bw$ is determined recursively from $(p_1,p_2,\dots,p_\ell)$ as in Notation \ref{Defn_of_phi_m}. The basic algorithm is then:
\begin{enumerate}
\item Factor $q_1=p_1=n_{\zeta_1}\dot s_1$ as $\mathbf k(q_1)\mathbf d(q_1)$ using Lemma \ref{k_and_d_lemma}.  As in Example \ref{Example1}, we determine that $z_1=\zeta_1$ and $\mathbf d(q_1)=\exp(\bar{\zeta}_1\check{E}_{\gamma_1})a(\zeta_1)^{-\check{H}_{\gamma_1}}$.
\item For $k$ from $2$ to $\ell$, first compute $q_k=\mathbf d(q_{k-1})p_k\in P_{s_k}$ and then rewrite $q_k$ in the form $n_{z_k}\dot s_kd_k$ for some $z_k\in \C$ and $d_k\in D$ depending on $\zeta_1,\dots,\zeta_k$.  This is possible because $\mathbf k(q_k)$ must have the form $\mathbf k(n_{z_k}\dot s_k)$ due to the fact that the diagram (\ref{basic_diagram}) commutes.  As a consequence, one obtains $\mathbf d(q_k)=\mathbf d(n_{z_k}\dot s_k)d_k$ for the next case.
\end{enumerate}
The process of rewriting $q_k$ in the form $q_k=n_{z_k}\dot s_kd_k$ will require several basic computational lemmas and, in practice, requires a detailed understanding of the structure of $G$.  For brevity, we introduce the notation
\[
\check{H}_\alpha=\frac{2}{\langle\langle\alpha,\alpha\rangle\rangle}H_\alpha
\]
for to the coroot associated to the root $\alpha$.  The proofs of the following three lemmas are straightforward computations.  Since our goal in this section is to illustrate their use in computing the change of variables, we omit their proofs.

\begin{lemma}\label{conjugation_by_a_lemma}
If $\alpha$ and $\beta$ are positive roots, then
\[
a^{-\check{H}_\alpha}\exp(u\check{E}_\beta)=\exp(a^{-2\frac{\langle\langle\alpha,\beta\rangle\rangle}{\langle\langle\alpha,\alpha\rangle\rangle}}uE_\beta)a^{-\check{H}_\alpha}
\]
for each $u\in\C$ and $a>0$ in $\R$.
\end{lemma}
\begin{lemma}\label{conjugation_by_n_lemma}
If $\alpha$ and $\beta$ are positive roots then
\[
\exp(u_1E_{\alpha})\exp(u_2E_{\beta})=\exp(u_2E_\beta)\exp(u_1u_2[E_\alpha,E_\beta])\exp(u_1E_{\alpha})
\]
for each $u_1,u_2\in\C$.
\end{lemma}
\begin{lemma}\label{conjugation_by_s_lemma}
If $s$ is a simple reflection and $\alpha$ is a root, then
\[
\dot s^{-1}a(u)^{-\check{H}_\alpha}\dot s=a(u)^{-\frac{\langle\langle s.\alpha,s.\alpha\rangle\rangle}{\langle\langle\alpha,\alpha\rangle\rangle}\check{H}_{s.\alpha}}
\]
for each $u\in\C$.
\end{lemma}
\begin{example}\label{Example2}
When $\ell(w)=2$, $\bw=(s_1,s_2)$ with $s_1\not=s_2$.  We set $p_1=n_{\zeta_1}\dot s_1$, $p_2=n_{\zeta_2}\dot s_2$ and $q_1=p_1$, $q_2=\mathbf d(q_1)p_2$.  We know that $z_1=\zeta_1$.  Following the algorithm, we compute
\begin{eqnarray*}
q_2
& = & \exp(\bar{\zeta}_1\check{E}_{\gamma_1})a(\zeta_1)^{-\check{H}_{\gamma_1}}\exp(\zeta_2\check{E}_{\gamma_2})\dot s_2 \\
& = & \exp(\bar{\zeta}_1\check{E}_{\gamma_1})\exp\left(a(\zeta_1)^{-2\frac{\langle\langle\gamma_1,\gamma_2\rangle\rangle}{\langle\langle\gamma_1,\gamma_1\rangle\rangle}}\zeta_2 \check{E}_{\gamma_2}\right)\dot s_2 a(\zeta_1)^{-\check{H}_{s_2.\gamma_1}}
\end{eqnarray*}
by Lemma \ref{conjugation_by_a_lemma} and Lemma \ref{conjugation_by_s_lemma}.
Now, we use Lemma \ref{conjugation_by_n_lemma} to move $\exp(\bar{\zeta}_1\check{E}_{\gamma_1})$ to the right, obtaining
\begin{equation}
q_2= \exp\left(a(\zeta_1)^{-2\frac{\langle\langle\gamma_1,\gamma_2\rangle\rangle}{\langle\langle\gamma_1,\gamma_1\rangle\rangle}}\zeta_2 \check{E}_{\gamma_2}\right)\dot s_2 d_2
\end{equation}
where
\begin{equation}
d_2=  \exp\left(a(\zeta_1)^{-2\frac{\langle\langle \gamma_1,\gamma_2\rangle\rangle}{\langle\langle \gamma_1,\gamma_1\rangle\rangle}}\bar{\zeta}_1\zeta_2(\dot s_2^{-1}.[\check{E}_{\gamma_1},\check{E}_{\gamma_2}])\right)\exp(\bar{\zeta}_1(\dot s_2^{-1}.\check{E}_{\gamma_1}))a(\zeta_1)^{-\check{H}_{s_2.\gamma_1}}\label{factor_d}
\end{equation}
by Lemma \ref{conjugation_by_n_lemma}. But $d_2\in D$ because $[E_{\gamma_1},E_{\gamma_2}]\in \g_{\gamma_1+\gamma_2}$, and thus $\dot s_2^{-1}.[E_{\gamma_1},E_{\gamma_2}]\in \g_{\gamma_1}\subset \n$ while $\dot s_2^{-1}.\check{E}_{\gamma_1}\in\n$ as well.   Thus,
$\mathbf k(q_2)=\mathbf k(n_{z_2}\dot s_2)=\mathbf k(\exp(z_2\check{E}_{\gamma_2})\dot s_2)$ where $z_2=a(\zeta_1)^{-2\frac{\langle\langle \gamma_1,\gamma_2\rangle\rangle}{\langle\langle \gamma_1,\gamma_1\rangle\rangle}}\zeta_2$.  Therefore,
\begin{equation}\label{lis2formulas}
z_1=\zeta_1 \text{ and }z_2=(1+|\zeta_1|^2)^{\frac{\langle\langle\gamma_1,\gamma_2\rangle\rangle}{\langle\langle \gamma_1,\gamma_1\rangle\rangle}}\zeta_2.
\end{equation}
\end{example}
\begin{remark}From (\ref{lis2formulas}) we see that
when the sequence $\bw=(s_1,s_2)$ of simple reflections is such that the associated simple positive roots $\gamma_1$ and $\gamma_2$ are orthogonal with respect to (the dual of) the Killing form, the change of variables reduces again to the identity map.   So, in that case, Lu's coordinates again agree with the standard holomorphic coordinates.  But when $\langle\langle \gamma_1,\gamma_2\rangle\rangle\not=0$, then $z_2$ depends on $\bar{\zeta}_1$ and thus Lu's coordinates are not holomorphic. This analysis shows that this change of character is not an artifact of different perspectives between the complex and real models of the Schubert cell, but an intrinsic difference due to the role of the real algebraic factorization maps and the nature of the sequence $\bw$.
\end{remark}
\begin{example}
We conclude with a more specific example of length 3.  Set $G=\mathrm{SL}(3,\C)$, $B$ equal to the set of upper triangular matrices, $H$ equal the set of diagonal matrices, and $K=\SU(3)$.  Let $\gamma_1=\lambda_1-\lambda_2$ and $\gamma_2=\lambda_2-\lambda_3$ where $\lambda_j$ is the linear function on $\mathfrak{sl}(3,\C)$ which selects the $(j,j)$ entry.  Then $\gamma_1$ and $\gamma_2$ are simple positive roots and we let $r_1$ and $r_2$ denote the corresponding simple reflections.

Let $w$ equal the longest element of the Weyl group.  The sequence $\bw=(s_1,s_2,s_3)=(r_1,r_2,r_1)$ is a reduced decomposition of $w$ and we let $\gamma_3=\gamma_1$ denote positive simple root associated to $s_3=r_1$.  The determination of $z_1$ and $z_2$ is the same as in Example \ref{Example2}.  In this case, $\langle\langle \gamma_j,\gamma_j\rangle\rangle=2$ for each $j=1,2,3$ while $\langle\langle \gamma_j,\gamma_{j+1}\rangle\rangle=-1$ for $j=1,2$.  This simplifies the powers of the function $a(\cdot)$ that appear in the formulas.  In particular, for $z_1$ and $z_2$, we have
\[
z_1=\zeta_1 \text{ and }z_2=\frac{\zeta_2}{\sqrt{1+|\zeta_1|^2}}
\]
from equation (\ref{lis2formulas}).

To determine $z_3$, we first simplify the factor $d_2$ from (\ref{factor_d}) in this specific setting.  Here, $\check{E}_{\gamma_1}=\check{E}_{\gamma_3}=E_{12}$ and $\check{E}_{\gamma_2}=E_{23}$ where $E_{ij}$ denotes the matrix with 1 in position $(i,j)$ and zeros elsewhere.  Then we can quickly verify
\[
[\check{E}_{\gamma_1},\check{E}_{\gamma_2}]=E_{13},\ \dot s_2^{-1}.[\check{E}_{\gamma_1},\check{E}_{\gamma_2}]=iE_{12}, \text{ and }\dot s_2^{-1}.\check{E}_{\gamma_1}=iE_{13}
\]
by direct computation. Furthermore, $\check{H}_{\gamma_1}=\check{H}_{\gamma_3}=E_{11}-E_{22}$ which we denote by $H_{12}$ and $\check{H}_{\gamma_2}=E_{22}-E_{33}$ which we denote by $H_{23}$, for brevity.  Then $s_2.\gamma_1=\gamma_1+\gamma_2$ and $s_3.\gamma_2=s_1.\gamma_2=\gamma_1+\gamma_2$. Set $H_{13}=E_{11}-E_{33}=\check{H}_{\gamma_1+\gamma_2}$. Then, specializing (\ref{factor_d}), we obtain
\[
d_2=  \exp\left(a(\zeta_1)\bar{\zeta}_1\zeta_2iE_{12}\right)\exp(\bar{\zeta}_1iE_{13})a(\zeta_1)^{-H_{13}}.
\]
Following the algorithm, we compute that
\[
\mathbf d(q_2)=\mathbf d(n_{z_2}\dot s_2)d_2=\exp(\bar{z}_2E_{23})a(z_2)^{-H_{23}}d_2
\]
by Lemma \ref{k_and_d_lemma}.  Then
\begin{eqnarray*}
q_3&=&\mathbf d(q_2)p_3 \\
& = &\mathbf d(\exp(z_2E_{23})\dot s_2)d_2 \,n_{\zeta_3}\dot s_3\\
& = & \exp(\bar{z}_2E_{23})a(z_2)^{-H_{23}}\exp(a(\zeta_1)i\bar{\zeta}_1\zeta_2 E_{12})\exp(i\bar{\zeta}_1E_{13})a(\zeta_1)^{-H_{13}}\exp(\zeta_3E_{12})\dot s_3.
\end{eqnarray*}
Our goal is to write this as $\exp(z_3E_{12})\dot s_3d_3$ for some $d_3\in D$ and a unique $z_3\in\C$.  To do that, we first move $a(\zeta_1)^{-H_{13}}$ to the right using Lemma \ref{conjugation_by_a_lemma} and Lemma \ref{conjugation_by_s_lemma} to write
\[
q_3=\exp(\bar{z}_2E_{23})a(z_2)^{-H_{23}}\exp(a(\zeta_1)i\bar{\zeta}_1\zeta_2 E_{12})\exp(i\bar{\zeta}_1E_{13})\exp(a(\zeta_1)^{-1}\zeta_3E_{12})\dot s_3a(\zeta_1)^{H_{23}}
\]
because $s_3.(\gamma_1+\gamma_2)=\gamma_2$.
Then we can interchange the factors $\exp(i\bar{\zeta}_1E_{13})$ and $\exp(a(\zeta_1)\zeta_3E_{12})$ because $E_{13}$ and $E_{12}$ commute.  Thus
\begin{eqnarray*}
q_3 & =& \exp(\bar{z}_2E_{23})a(z_2)^{-H_{23}}\exp(a(\zeta_1)(i\bar{\zeta}_1\zeta_2 E_{12})\exp(a(\zeta_1)^{-1}\zeta_3E_{12})\dot s_3\exp(\bar{\zeta}_1E_{23})a(\zeta_1)^{H_{23}} \\
& = & \exp(\bar{z}_2E_{23})a(z_2)^{-H_{23}}\exp(a(\zeta_1)(i\bar{\zeta}_1\zeta_2 +a(\zeta_1)^{-2}\zeta_3)E_{12})\dot s_3\exp(\bar{\zeta}_1E_{23})a(\zeta_1)^{H_{23}}.
\end{eqnarray*}
Next, we move the term $a(z_2)^{-H_{23}}$ to the right using Lemma \ref{conjugation_by_a_lemma} and Lemma \ref{conjugation_by_s_lemma}.  This gives
\begin{eqnarray*}
q_3 & =&  \exp(\bar{z}_2E_{23})\exp(a(z_2)a(\zeta_1)(i\bar{\zeta}_1\zeta_2 +a(\zeta_1)^{-2}\zeta_3)E_{12})a(z_2)^{-H_2}\dot s_3 \exp(\bar{\zeta}_1E_{23})a(\zeta_1)^{H_1}\\
& =&  \exp(\bar{z}_2E_{23})\exp(a(z_2)a(\zeta_1)(i\bar{\zeta}_1\zeta_2 +a(\zeta_1)^{-2}\zeta_3)E_{12})\dot s_3a(z_2)^{H_1} \exp(\bar{\zeta}_1E_{23})a(\zeta_1)^{H_1}
\end{eqnarray*}
since $s_3=r_1$.  Finally, we move the term $\exp(\bar{z}_2E_{23})$ to the right using Lemma \ref{conjugation_by_n_lemma} to find
\[
q_3 = \exp(a(z_2)a(\zeta_1)(i\bar{\zeta}_1\zeta_2 +a(\zeta_1)^{-2}\zeta_3)E_{12})\dot s_3d_3
\]
where $d_3\in D$.  Thus $\mathbf k(q_3)=\mathbf k(n_{z_3}\dot s_3)$ where $z_3=a(z_2)a(\zeta_1)(i\bar{\zeta}_1\zeta_2 +a(\zeta_1)^{-2}\zeta_3)$.
The expression $a(z_2)a(\zeta_1)=a(a(\zeta_1)\zeta_2)a(\zeta_1)$ simplifies to $(1+|\zeta_1|^2+|\zeta_2|^2)^{-1/2}$.  Thus, in total,
\begin{equation}\label{SL3_change}
z_1=\zeta_1,\ z_2=\frac{\zeta_2}{\sqrt{1+|\zeta_1|^2}},\text{ and }z_3=\frac{i\bar{\zeta}_1\zeta_2+\zeta_3(1+|\zeta_1|^2)}{\sqrt{1+|\zeta_1|^2+|\zeta_2|^2}}.
\end{equation}
\end{example}
\begin{remark}
On pg. 365 of \cite{Lu}, Lu considers this same example and computes the change of variables between the coordinates $z_1,z_2,z_3$ and variables $u_1,u_2,u_3$.  This second set of variables should not be confused with $\zeta_1,\zeta_2,\zeta_3$.  The relationship between them is given by the diffeomorphism $M_\bw\colon N_\bw\to N_w$.  Specifically,
\[
\begin{bmatrix}
1 & u_1 & u_3 \\
0 & 1 & u_2 \\
0 & 0 & 1
\end{bmatrix}
=
n_{\zeta_1}\dot s_1n_{\zeta_2}\dot s_2n_{\zeta_3}\dot s_3\dot w^{-1}=\begin{bmatrix}
1 & \zeta_1 & i\zeta_2+\zeta_1\zeta_3 \\
0 & 1 & \zeta_3 \\
0 & 0 & 1
\end{bmatrix}
\]
which means $u_1=\zeta_1$, $u_2=\zeta_3$, and $u_3=i\zeta_2+\zeta_1\zeta_3$.  One can confirm that with these substitutions, our change of variables (\ref{SL3_change}) agrees with the one given at the top of page 366 of \cite{Lu} obtained via the Gram-Schmidt process.
\end{remark}

\end{document}